\documentclass[10pt]{amsart}

\usepackage{amsmath,amsfonts, latexsym,graphicx, footmisc, amssymb, mathtools, enumerate}

\usepackage{hyperref}
\hypersetup{colorlinks=true, urlcolor=blue, citecolor=blue, linkcolor=blue}

\newcommand{\U}{{\mathcal U}}
\newcommand{\0}{{\mathbf 0}}
\newcommand{\C}{{\mathbb C}}
\newcommand{\Z}{{\mathbb Z}}

\newcommand{\dm}{\operatorname{dim}}

\newtheorem{defn0}{Definition}[section]
\newtheorem{prop0}[defn0]{Proposition}
\newtheorem{conj0}[defn0]{Conjecture}
\newtheorem{thm0}[defn0]{Theorem}
\newtheorem{lem0}[defn0]{Lemma}
\newtheorem{corollary0}[defn0]{Corollary}
\newtheorem{example0}[defn0]{Example}
\newtheorem{remark0}[defn0]{Remark}
\newtheorem{question0}[defn0]{Question}
\newtheorem{exercise0}[defn0]{Exercise}

\newenvironment{defn}{\begin{defn0}}{\end{defn0}}
\newenvironment{prop}{\begin{prop0}}{\end{prop0}}

\newenvironment{thm}{\begin{thm0}}{\end{thm0}}

\newenvironment{rem}{\begin{remark0}\rm}{\end{remark0}}

\newcommand{\propref}[1]{Proposition~\ref{#1}}
\newcommand{\thmref}[1]{Theorem~\ref{#1}}

\newcommand{\secref}[1]{Section~\ref{#1}}

\newcommand{\mbf}[1]{{\mathbf #1}}

\title{Control by the lowest degree vanishing cycles}


\author{David B. Massey}

\date{}

\begin{document}

\begin{abstract} Given the germ of an analytic function on affine space with a smooth critical locus, we prove that the constancy of the stalk cohomology of the Milnor fiber in lowest degree off a codimension two subset of the critical locus implies that the vanishing cycles are concentrated in lowest degree and are constant. \end{abstract}

\maketitle




\section{Introduction} Suppose that $\U$ is a non-empty open neighborhood of the origin in $\C^{n+1}$, where $n\geq 1$, and let $f:(\U, \0)\rightarrow (\C,0)$ be a nowhere locally constant, complex analytic function. Then $V(f)=f^{-1}(0)$ is a hypersurface in $\U$ of pure dimension $n$, where for convenience we have assumed that $\0\in V(f)$. 

Let $s$ denote the dimension, $\dim_\0\Sigma f$,  of  the critical locus of $f$ at the origin. As is well-known, the Curve Selection Lemma implies that, near $\0$, $\Sigma f\subseteq V(f)$, and we choose $\U$ small enough so that this containment holds everywhere in $\U$ and so ever irreducible component of $\Sigma f$ in $\U$ contains the origin.

Suppose that $\mbf p\in\Sigma f$, and let $d:=\dim_{\mbf p}\Sigma f$; it is well-known (see \cite{katomatsu}) that the reduced integral homology, $\widetilde H_k(F_{f, \mbf p}; \Z)$, of $F_{f, \mbf p}$ can be non-zero only for $n-d\leq k\leq n$, and is free Abelian in degree $n$.  Cohomologically, this implies that $\widetilde H^k(F_{f, \mbf p}; \Z)$ can be non-zero only for $n-d\leq k\leq n$, and is free Abelian in degree $n-d$.

\bigskip

The most basic form of our main result is:

\medskip

\noindent {\bf Theorem}: (Main Theorem - basic form)\  {\it Suppose that, at the origin, $\Sigma f$ is smooth of dimension $s$. Suppose also that there exists an analytic subset $Y\subseteq\Sigma f$ such that $\dim_\0 Y\leq s-2$ and, for all $\mbf p\in \Sigma f\backslash Y$, $\widetilde H^{n-s}(F_{f, \mbf p}; \Z)\cong \Z^{\mu_0}$, where $\mu_0$ is independent of $\mbf p$. 

Then, $f$ defines an $s$-dimensional family of isolated singularities which is $\mu$-constant (i.e., has constant Milnor number). In particular, this implies, for all $\mbf p\in\Sigma f$ near $\0$, $\widetilde H^k(F_{f, \mbf p}; \Z)$ is zero except when $k=n-s$ and $\widetilde H^{n-s}(F_{f, \mbf p}; \Z)\cong \Z^{\mu_0}$. 
}

\bigskip

We  interpret this result in terms of the complex of sheaves of vanishing cycles and, of course, can combine it with the result of L\^e-Ramanujam \cite{leramanujam}  to reach a conclusion about the ambient topological-type of the hypersurface along $\Sigma f$. 

From the statement of the main theorem, one might expect that the proof uses known properties of one or more of the derived category, the category of perverse sheaves, the Decomposition Theorem, A'Campo's result on the Lefschetz number of the monodromy, the weight/nilpotent filtration on the nearby cycles, mixed Hodge modules, etc. In fact, we use none of these; the crux of the proof uses the main result of ours with L\^e in \cite{lemassey}, the proof of which involves the geometry of the relative polar curve and distinguished bases for the vanishing cycles in the isolated critical point case. 

In \secref{sec:muconst}, we recall various equivalences for $\mu$-constant families, as described in \cite{lemassey}. In \secref{sec:lecycles}, we  recall basic properties of L\^e cycles that we need, recall our main result with L\^e from \cite{lemassey}, and prove the  main theorem.  In the final section of this paper, we discuss possible generalizations and questions that naturally arise.

\section{Families with constant Milnor number}\label{sec:muconst}

This section is essentially a summary of Section 2 of \cite{lemassey}.

\smallskip

We continue with the notation from the introduction: $\U$ is a non-empty open neighborhood of the origin in $\C^{n+1}$, where $n\geq 1$,  $f:(\U, \0)\rightarrow (\C,0)$ is a nowhere locally constant, complex analytic function, and $s=\dim_\0\Sigma f$.

\smallskip

We will use $\mathbf x:=(x_0, \dots, x_n)$ to denote the standard coordinate functions on $\C^{n+1}$. We will use $\mathbf z:=(z_0, \dots, z_n)$ to denote arbitrary analytic local coordinates on $\mathcal U$ near the origin. All of our constructions and results will depend only on the linear part of the coordinates $\mathbf z$; hence, when we say that the $\mathbf z$ are chosen generically, we mean that the linear part of $\mathbf z$ consists of a generic linear combination of $\mathbf x$ (generic in $\operatorname{PGL}(\mathbb C^{n+1})$).

\smallskip

We wish to consider families of singularities. Fix a set of local coordinates $\mathbf z$ for $\mathcal U$ at the origin. Let $G:=(z_0, \dots, z_{s-1})$. If $\mathbf q\in\mathcal U$, we define $f_{\mathbf q}:=f_{|_{G^{-1}(G(\mathbf q))}}$.

\medskip

\begin{defn}\label{def:simplemuconst}
We say that  $f_{\mathbf q}$  is  a simple $\mu$-constant family  at the origin if and only if , at the origin
\begin{itemize}
\item $f_{\mathbf 0}$ has an isolated critical point, 
\smallskip
\item $\Sigma f$ is smooth, 
\smallskip
\item $G_{|_{\Sigma f}}$ has a regular point, and, 
\end{itemize}
\noindent for all $\mathbf q\in\Sigma f$ close to the origin, the Milnor number, $\mu_{\mathbf q}(f_{\mathbf q})$, of $f_{\mbf q}$ at $\mbf q$ is independent of $\mathbf q$.
\end{defn}

A simple $\mu$-constant family may seem like too strong a notion of a ``$\mu$-constant family''. However, as we shall see in \thmref{thm:milnorequi} below, all other reasonable concepts of $\mu$-constant families are equivalent. First, we need some notation.

\bigskip

Suppose that $\dm_{\mathbf 0}\Sigma(f_{\mathbf 0})=0$. Then, the analytic cycle 
$$\left[V\Big(z_0, \dots, z_{s-1}, \frac{\partial f}{\partial z_s}, \dots, \frac{\partial f}{\partial z_n}\Big)\right]$$
 has the origin as a $0$-dimensional component, and $[\mathbf 0]$ appears in this cycle with multiplicity $\mu_{\mathbf 0}(f_{\mathbf 0})$. Thus, at the origin, $C:=\left[V\Big( \frac{\partial f}{\partial z_s}, \dots, \frac{\partial f}{\partial z_n}\Big)\right]$ is purely $s$-dimensional and is properly intersected by $[V(z_0, \dots, z_{s-1})]$. 
 
 Let $\Gamma^s_{f, \mathbf z}$ denote the sum of the components of $C$ which are not contained in $\Sigma f$, and let $\Lambda^s_{f, \mathbf z}:= C-\Gamma^s_{f, \mathbf z}$. The cycles $\Gamma^s_{f, \mathbf z}$and $\Lambda^s_{f, \mathbf z}$  are, respectively, the $s$-dimensional polar cycle and $s$-dimensional L\^e cycle; see \cite{lecycles}. It follows at once that, for all $\mbf q\in\Sigma f$ near 
 $\0$, 
 
$$
\mu_{\mathbf q}(f_{\mathbf q}) = \big(\Gamma^s_{f, \mathbf z}\cdot V(z_0-q_0, \dots, z_{s-1}-q_{s-1})\big)_{\mathbf q}+\big(\Lambda^s_{f, \mathbf z}\cdot V(z_0-p_0, \dots, z_{s-1}-p_{s-1})\big)_{\mathbf q}. \leqno{(\dagger)}$$

Note that $\Gamma^s_{f, \mathbf z}=0$ is equivalent to the equality of sets $\displaystyle \Sigma f= V\Big(\frac{\partial f}{\partial z_s}, \dots, \frac{\partial f}{\partial z_n}\Big)$.

\smallskip

For each $s$-dimensional component, $\nu$, of $\Sigma f$, for a generic point $\mbf p\in\nu$, for a generic codimension $s$ (in $\U$) affine linear subspace, $N$ (a normal slice),  containing $\mbf p$, the function $f_{|_N}$ has an isolated critical point at $\mbf p$ and the Milnor number at $\mbf p$ is independent of the choices;  we let ${\stackrel{\circ}{\mu}}_\nu$ denote this common value. 

Then $\Lambda^s_{f, \mathbf z}=\sum_\nu {\stackrel{\circ}{\mu}}_\nu[\nu]$, where the sum is over the $s$-dimensional components $\nu$ of $\Sigma f$, and, by definition, $\lambda^s_{f, \mathbf z}(\mathbf 0)= \big(\Lambda^s_{f, \mathbf z}\cdot V(z_0, \dots, z_{s-1})\big)_{\mathbf 0}$. Therefore, the $s$-dimensional L\^e number \cite{lecycles}, $\lambda^s_{f, \mathbf z}(\mathbf 0)$,  at the origin is defined, and 
$$\lambda^s_{f, \mathbf z}(\mathbf 0)=\sum_\nu {\stackrel{\circ}{\mu}}_\nu\big(\nu\cdot V(z_0, \dots, z_{s-1})\big)_{\mathbf 0}.
$$
 If the coordinates $(z_0, \dots, z_{s-1})$ are sufficiently generic, then $\lambda^s_{f, \mathbf z}(\mathbf 0)$ obtains its minimum value of $\sum_\nu {\stackrel{\circ}{\mu}}_\nu{\operatorname{mult}}_{\mathbf 0}\nu$; we denote this generic value by $\lambda^s_{f}(\mathbf 0)$ (with no subscript by the coordinates). 

\medskip

Note that $\dm_{\mathbf 0}\Sigma(f_{\mathbf 0})=0$ implies that, for all $\mbf q\in\Sigma f$ near $\0$, $\dim_{\mbf q}\Sigma(f_{\mbf q})=0$.

\medskip

There is one more piece of preliminary notation that we need. Consider the blow-up of $\mathcal U$ along the Jacobian ideal, $J(f)$ of $f$, i.e., $B:={\operatorname{Bl}}_{J(f)}\mathcal U$. This blow-up naturally sits inside $\mathcal U\times \mathbb P^n$. Thus, the exceptional divisor $E$ of the blow-up is a cycle in $\mathcal U\times \mathbb P^n$. 

\medskip

We now give  a number of equivalent characterizations of {\bf $\mu$-constant families}; this is a combination of Theorem 2.3 and Corollary 5.4 of \cite{lemassey}.

\medskip

\begin{thm}\label{thm:milnorequi} Let $\mathbf z$ be local coordinates for $\mathcal U$ at the origin such that $\dm_{\mathbf 0}\Sigma(f_{\mathbf 0})=0$. Then, the following are equivalent: 
\begin{enumerate}[1{.}]
\item For all $\mathbf q\in\Sigma f$ near the origin, $\mu_{\mathbf 0}(f_{\mathbf 0}) = \mu_{\mathbf q} (f_{\mathbf q})$.
\smallskip
\item $\mu_{\mathbf 0}(f_{\mathbf 0})=\lambda^s_f(\mathbf 0)$.
\smallskip
\item $f_{\mathbf q}$ is a simple $\mu$-constant family.
\smallskip

\item $\mu_{\mathbf 0}(f_{\mathbf 0})=\lambda^s_{f, \mathbf z}(\mathbf 0)$.
\smallskip

\item $\Gamma^s_{f, \mathbf z}=0$.
\end{enumerate}

Furthermore, if $n-s\neq 2$, then 1), 2), 3), 4), and 5) above hold if and only if the local, ambient, topological-type of $V(f_{\mathbf q})$ at $\mathbf q$ is independent of the point $\mathbf q\in\Sigma f$ near the origin.

\medskip

\noindent In addition, the following are equivalent:
\begin{enumerate}[a{.}]
\item There exist coordinates $\mathbf z$ such that 1), 2), 3), 4), and 5) above hold.
\smallskip

\item Near the origin, $\Sigma f$ is smooth and $(\mathcal U-\Sigma f, \Sigma f)$ is an $a_f$ stratification, i.e., for all $\mathbf p\in\Sigma f$ near the origin,  for every limiting tangent space, $\mathcal T_{\mathbf p}$, from level hypersurfaces of $f$ approaching $\mathbf p$ , $T_{\mathbf p}(\Sigma f)\subseteq \mathcal T_{\mathbf p}$.
\smallskip

\item $\Sigma f$ is smooth at the origin, and over an open neighborhood of the origin, the exceptional divisor, $E$, as a set, is equal to the projectivized conormal variety to $\Sigma f$ and, hence, as cycles $E=\mu\left[T^*_{\Sigma f}\,\U\right]$ for some positive integer $\mu$.
\smallskip

\item For generic $\hat{\mathbf z}$, $\Gamma^s_{f, \hat{\mathbf z}} =0$ near the origin.
\smallskip

\item $\Sigma f$ is smooth at the origin and, for all local coordinates $\hat{\mathbf z}$ such that $V(\hat z_0, \dots, \hat z_{s-1})$ transversely intersects $\Sigma f$ at the origin, $\hat f_{\mathbf q}$ is a simple $\mu$-constant family.
\smallskip

\item $\Sigma f$ is smooth at the origin and, for all $\mbf q$ near the origin, the non-zero reduced cohomology of $F_{f, \mbf q}$ is concentrated in degree $n-s$, and $\widetilde H^{n-s}(F_{f, \mbf q}; \Z)\cong \Z^\mu$, where $\mu$ is independent of $\mbf q$.
\smallskip

\item  $\Sigma f$ is smooth at the origin and the constructible complex of shifted, restricted vanishing cycles $\big(\phi_f[-1]\Z^\bullet_\U[n+1]\big)_{|_{\Sigma f}}$ is isomorphic in the derived category (or category of perverse sheaves) to a shifted constant sheaf $\big(\Z^\mu\big)^\bullet_{\Sigma f}[s]$ for some positive integer $\mu$.

\end{enumerate}

\smallskip

\noindent In addition, the $\mu$ in (c), (f), and (g) equals the $\mu_\0(f_\0)$ in (1).

\end{thm}

\medskip

Of course, we make the following definition:

\begin{defn} We say that $f$ {\bf defines a $\mu$-constant family} at/near the origin provided that the equivalent conditions in (a)-(g) of \thmref{thm:milnorequi} hold.
\end{defn}

\smallskip

Of course, we say that $f$ defines a $\mu$-constant family near an arbitrary point $\mbf p\in\Sigma f$ provided that conditions in (a)-(g) of \thmref{thm:milnorequi} hold with the origin replaced with $\mbf p$.

\medskip

\section{L\^e Cycles and the Main Theorem}\label{sec:lecycles}

We continue with $f$ as in the previous two sections; in particular,  $\Sigma f\subseteq V(f)$ and $s:=\dim\Sigma f=\dim_\0\Sigma f$. The reader is referred to  \cite{lecycles} and  \cite{nonisolle} for details of L\^e cycles and L\^e numbers, but we shall summarize needed properties here. Recall that the cycle $\Gamma_{f, \mbf z}^s$ was defined in the previous section.

\medskip

\begin{prop}\label{prop:lecycles}
For a generic linear choice of coordinates ({\it prepolar coordinates}) $\mbf z$ for $\C^{n+1}$, there exists an open neighborhood of the origin (which we call $\U$ again) such that the {\it L\^e cycles} $\Lambda_{f, \mbf z}^s$, \dots, $\Lambda_{f, \mbf z}^1$, $\Lambda_{f, \mbf z}^0$ are defined inside $\U$ and have the following properties:

\begin{enumerate}
\item Each $\Lambda_{f, \mbf z}^k$ is a purely $k$-dimensional analytic effective cycle ({\bf not} a cycle class) in $\Sigma f$.
\smallskip
\item $\Sigma f=\bigcup_{k=0}^s\big|\Lambda_{f, \mbf z}^k\big|$, where $\big|\cdot\big|$ denotes the underlying set, and so every $k$-dimensional component of $\Sigma f$ is contained in $\big|\Lambda_{f, \mbf z}^k\big|$.
\smallskip
\item If $s>0$, then $\big|\Gamma^s_{f, \mbf z}\big|\cap\Sigma f =\bigcup_{k=0}^{s-1}\big|\Lambda_{f, \mbf z}^k\big|$. In particular, if the cycle $\Gamma_{f, \mbf z}^s=0$, then, for $0\leq k\leq s-1$, $\Lambda_{f, \mbf z}^k=0$.
\smallskip
\item For all $\mbf p\in\Sigma f$, for all $k$ such that $0\leq k\leq s$, $\Lambda_{f, \mbf z}^k$ properly intersects $V(z_0-p_0, \dots, z_{k-1}-p_{k-1})$ at $\mbf p$ (when $k=0$, the intersection is with $\U$). The $k$-th L\^e number of $f$ at $\mbf p$ with respect to $\mbf z$, $\lambda_{f, \mbf z}^k(\mbf p)$,  is defined to be the intersection number 
$$
\left(\Lambda_{f, \mbf z}^k\cdot V(z_0-p_0, \dots, z_{k-1}-p_{k-1})\right)_{\mbf p}.
$$
\smallskip
\item For each $\mbf p\in\Sigma f$, letting $d:=\dim_{\mbf p}\Sigma f$, there is a chain complex of free Abelian $\Z$-modules
$$
0\rightarrow \Z^{\lambda^{d}_{f, \mbf z}(\mbf p)}\rightarrow  \Z^{\lambda^{d-1}_{f, \mbf z}(\mbf p)}\rightarrow\cdots\rightarrow  \Z^{\lambda^{1}_{f, \mbf z}(\mbf p)}\rightarrow  \Z^{\lambda^{0}_{f, \mbf z}(\mbf p)}\rightarrow 0,
$$
where the cohomology at the $\lambda^{k}_{f, \mbf z}(\mbf p)$ term is  isomorphic to $\widetilde H^{n-k}(F_{f, \mbf p}; \Z)$.
\smallskip
\end{enumerate}
\end{prop}

\bigskip

We also recall Theorem 5.3 of \cite{lemassey}:

\begin{thm}\label{thm:lemassey} \textnormal{(L\^e-Massey)}\  Let $\mbf p\in\Sigma f$, let $d=\dim_{\mbf p}\Sigma f$, and suppose that $\mbf z$ are prepolar coordinates for $f$ at $\mbf p$. Finally, suppose that $\widetilde H^{n-d}(F_{f, \mbf p}; \Z)\cong \Z^{\lambda^d_{f, \mbf z}(\mbf p)}$. Then $f$ defines a $\mu$-constant family near $\mbf p$.
\end{thm}

\bigskip

Now we prove the main theorem; it is essentially an application of \thmref{thm:lemassey}, but we find the statement surprising.

\begin{thm}\label{thm:main}  Suppose $\Sigma f$ is purely $s$-dimensional and that every irreducible component of $\Sigma f$ is smooth at $\0$. Suppose also that there exists an analytic subset $Y\subseteq\Sigma f$ of dimension at most $s-2$ such that $\Sigma f\backslash Y$ is smooth and, for all $\mbf p\in \Sigma f\backslash Y$, $\widetilde H^{n-s}(F_{f, \mbf p}; \Z)\cong \Z^{\mu_0}$, where $\mu_0$ is independent of $\mbf p$. 

Then $f$ defines a $\mu$-constant family near $\0$ with constant Milnor number $\mu_0$ (in particular, at $\0$, $\Sigma f$ has a single smooth irreducible component).
\end{thm}
\begin{proof} Let $\mbf z$ be prepolar coordinates for $f$ at $\0$ such that $V(z_0, \dots, z_{s-1})$ transversely intersects each irreducible component of $\Sigma f$ at $\0$. We first wish to show that $\Lambda^{s-1}_{f, \mbf z}$ is zero near $\0$. Suppose that it is not.

Let $\mbf p\in |\Lambda^{s-1}_{f, \mbf z}|\backslash Y$ be such that $\mbf p$ is a smooth point of $|\Lambda^{s-1}_{f, \mbf z}|$ and is not in a smaller-dimensional L\^e cycle. Furthermore, we choose $\mbf p$ close enough to $\0$ so that $V(z_0-p_0, \dots, z_{s-1}-p_{s-1})$ transversely intersects $\Sigma f$ at $\mbf p$. 

Note that  (5) of \propref{prop:lecycles} implies that, for $\mbf q\in \Sigma f\backslash |\Lambda^{s-1}_{f, \mbf z}|$ near $\mbf p$, $\lambda^s_{f, \mbf z}(\mbf q)=\mu_0$. But $\Sigma f$ is smooth at $\mbf p$ and transversely intersected by $V(z_0-p_0, \dots, z_{s-1}-p_{s-1})$; thus $\lambda^s_{f, \mbf z}(\mbf p)=\mu_0$. Thus, our hypothesis implies that
$$
\widetilde H^{n-s}(F_{f, \mbf p}; \Z)\cong \Z^{\mu_0}=\Z^{\lambda^s_{f, \mbf z}(\mbf p)}
$$
and so, by \thmref{thm:lemassey}, $f$ defines a $\mu$-constant family near $\mbf p$. Now, by (e) and (5) from \thmref{thm:milnorequi}, $\Gamma^s_{f, \mbf z}$ is $0$ near $\mbf p$, which, by Property (3) above of L\^e cycles, implies that $\Lambda^{s-1}_{f, \mbf z}$ is zero near $\mbf p$; a contradiction of the choice of $\mbf p$.

Therefore $\Lambda^{s-1}_{f, \mbf z}$ is zero near $\0$, which means $\lambda^{s-1}_{f, \mbf z}(\0)=0$. Now (5) of \propref{prop:lecycles} tells us that $\widetilde H^{n-s}(F_{f, \0}; \Z)\cong \Z^{\lambda^s_{f, \mbf z}(\0)}$ and thus  \thmref{thm:lemassey} yields that $f$ defines a $\mu$-constant family near $\0$.
\end{proof}

\medskip

\begin{rem} Note that the assumption that $\widetilde H^{n-s}(F_{f, \mbf p}; \Z)\cong \Z^{\mu_0}$ is independent of the chosen $\mbf p\in \Sigma f\backslash Y$ is {\it a priori} weaker than saying that cohomology sheaf of the shifted, restricted vanishing cycles $\big(\phi_f[-1]\Z^\bullet_\U[n+1]\big)_{|_{\Sigma f}}$ in degree $-s$ is locally constant on $\Sigma f\backslash Y$.
\end{rem}

\section{Remarks and Questions}\label{sec:rem}

The most basic statement of the main theorem -- that, if $f$ has a smooth $s$-dimensional critical locus and the shifted vanishing cycles in degree $-s$ have constant stalk cohomology off a set of codimension 2, then the shifted vanishing cycles on all of $\Sigma f$ consist merely of a shifted constant sheaf  -- is a surprising result which in no way refers to L\^e cycles. This result really does seem as though it should be approachable through standard high-powered techniques and theorems about perverse sheaves, vanishing or nearby cycles, weight filtrations, the Decomposition Theorem, etc. And perhaps such a proof exists and would lead to a generalization of \thmref{thm:main}, but we do not see how to produce such a proof or generalization.

\medskip

One could hope to generalize \thmref{thm:main} by first proving a generalization of \thmref{thm:lemassey}. Perhaps the hypothesis that $\widetilde H^{n-d}(F_{f, \mbf p}; \Z)\cong \Z^{\lambda^d_{f, \mbf z}(\mbf p)}$ could be replaced with $\widetilde H^{n-k}(F_{f, \mbf p}; \Z)\cong \Z^{\lambda^k_{f, \mbf z}(\mbf p)}$ for some $k<d$ or perhaps one could use the hypothesis that one of the maps in the L\^e number chain complex from (5) of \propref{prop:lecycles}, other than $\Z^{\lambda^{d}_{f, \mbf z}(\mbf p)}\rightarrow  \Z^{\lambda^{d-1}_{f, \mbf z}(\mbf p)}$, is zero. However, aside from trivial generalizations, we do not see such a result.

\medskip

Finally, we mention that we originally hoped that Proposition 1.31 of \cite{lecycles} would enable us to produce a generalization of \thmref{thm:main}. That proposition says that, for prepolar coordinates $\mbf z$ at a point $\mbf p\in\Sigma f$, if pairs of distinct irreducible germs of $\Sigma f$ intersect in dimension at most $k-1$ at $\mbf p$ and $\lambda^k_{f, \mbf z}(\mbf p)=0$, then, for all $j\leq k$, $\lambda^j_{f, \mbf z}(\mbf p)=0$ and so, by (5) of \propref{prop:lecycles},  $\widetilde H^{n-j}(F_{f, \mbf p}; \Z)=0$ for $j\leq k$. 

Again, we have yet to see how this leads to a non-trivial generalization of  \thmref{thm:main} or \thmref{thm:lemassey}.

\bibliographystyle{plain}

\bibliography{Masseybib}

\end{document}